\documentclass[preprint,11pt]{amsart}
\usepackage{mathrsfs}
\usepackage{amsmath}
\usepackage{amssymb}

\usepackage{color}

\newtheorem{theorem}{Theorem}[section]
\newtheorem{lemma}[theorem]{Lemma}
\theoremstyle{definition}
\newtheorem{definition}[theorem]{Definition}
\newtheorem{example}[theorem]{Example}

\newtheorem{observation}[theorem]{Observation}
\theoremstyle{remark}
\newtheorem{remark}[theorem]{Remark}
\theoremstyle{corollary}
\newtheorem{corollary}[theorem]{Corollary}
\theoremstyle{proposition}
\newtheorem{proposition}[theorem]{Proposition}

\newcommand{\R}{\mathbb{R}}
\newcommand{\Cn}{\mathbb{C}^n}
\newcommand{\C}{\mathbb{C}}

\newcommand{\Z}{\mathbb{Z}}

\newcommand{\abs}[1]{\left|#1\right|}
\newcommand{\norm}[1]{\left\|#1\right\|}
\begin{document}

\title{An algebra of polyanalytic functions}
\author{Abtin Daghighi} 
\author{Paul M. Gauthier} 

%


%
%

\begin{abstract}
The most important uniform algebra is the family of continuous functions on a compact subset $K$ of the complex plane $\C$ which are analytic on the interior $int(K).$ For compact sets $K$ which are {\em regular} (i.e. $K=\overline{int(K)}$ and for polyanalytic functions, we introduce analogous spaces, which are Banach spaces with respect to the sup-norm, but are not closed with respect to the usual pointwise multiplication. We shall introduce a multiplication on these spaces and investigate the resulting algebras.
\end{abstract}

\maketitle

\section{Introduction}
 
Function algebras, in particular in  the context of Banach and uniform algebras, have many useful applications in the theory of analytic functions, see e.g.\ Alexander \& Wermer \cite{wermer}, Gamelin \cite{gamelinbok}, Stout \cite{stout} and Browder \cite{Browderbok}. 
There currently do not exist natural algebra structures
for function spaces that include only polyanalytic functions of a fixed finite order. 
The purpose of this paper is to introduce a multiplication that allows for the construction of an algebra of $q$-analytic functions, for any fixed $q\in \Z_+,$ closed with respect to the sup-norm.
For $q=1$ this reduces precisely to the case of analytic functions.
The procedure will naturally give rise to a generalized notion of uniform algebras. The main result of the paper is the verification of the fact that 
we in fact obtain an algebra which is also a Banach space with respect to the sup-norm. This we do in Section 2 and, following it, Section 3 and Section 4 deal with the associated generalization of uniform algebras establishing some analogues to those from the standard theory of complex function algebras.
\\
\\
In this paper, all vector spaces considered will be complex vector spaces. By a continuous function on a subset of a vector space, we mean a continuous  complex-valued function. When we speak of an analytic mapping from an open subset of one vector space to another vector space, we mean a complex-analytic mapping. 

\begin{definition} 
Let $\Omega\subseteq\C$ be an open subset. A complex function $f$ is called {\em $q$-analytic}
on $\Omega$ if it satisfies $\partial_{\bar{z}}^q f=0$ on $\Omega$ in the weak sense.
It follows that $q$-analyticity is a local property. 
We can equivalently (see e.g.\ Balk \cite{balk}) define $f$ to 
be $q$-analytic
on $\Omega$ if 
$f$ can be represented in the form
\begin{equation}\label{eq1}
	f(z)=\sum_{j=0}^{q-1}a_j(z)\bar{z}^j, \quad z\in \Omega,
\end{equation} 
where $a_j, j=0,\ldots , q-1,$ are analytic functions on $\Omega.$  
The functions $a_j$ are uniquely determined by $f$ and
are called the {\em analytic components} (or {\em holomorphic components}) of $f$.
We denote by $\mbox{PA}_q(\Omega)$ the set of $q$-analytic functions on $\Omega.$

Let us say that a function $f$ defined on an open set $\Omega$ is polyanalytic if it is $q$-analytic for some $q = 1,2,\ldots$ and let us denote by 
$$ 
	\mbox{PA}(\Omega) = \bigcup_{q=1}^\infty \mbox{PA}_q(\Omega)
$$
the family of polyanalytic functions on $\Omega.$ 

We shall say that a function $f$ defined on an open set  $\Omega$ is  {\em countably polyanalytic} (CA) at a point $a$ if, in some neighbourhood $U(a),$ it is representable as a uniformly convergent series
$$
	f(z) = \sum_{k=0}^\infty h_k(z:a)(\overline z-\overline a)^k,  
$$ 
where all $h_k$ are analytic in $U(a).$ Moreover, we shall say that $f$ is  countably polyanalytic on $\Omega$ if it is  countably polyanalytic at each point of $\Omega.$ We denote by $CA(\Omega)$ the space of all countably  polyanalytic functions on $\Omega.$ For a general introduction to polyanalytic functions, see Balk \cite{balk}. 

We have 
$$
\mbox{PA}_q(\Omega) \subset \mbox{PA}(\Omega) \subset CA(\Omega).
$$


If $f\in CA(\Omega),$ we may consider $f$ as a function of two complex variables $z$ and $w$ in a natural manner. For $\Omega^*=\{\overline z:z\in \Omega\}$ consider the open set $\Omega\times\Omega^*\subset \C^2$ and let $G=G_\Omega\subset \Omega\times\Omega^*$ be the graph of the function $w=\overline z$ over the set $\Omega.$ To $f\in CA(\Omega),$ we shall associate  a function $F(z,w)$ holomorphic in a neighbourhood of  $G_\Omega.$ For each $a\in \Omega,$ there is an open disc $U_a$ centered at $a$ in which $f$ has a representation (see \cite{balk})
$$
	f(z)=\sum_{k,\nu=0}^\infty c_{k,\nu}(\overline z-\overline a)^k(z-a)^\nu, \quad c_{k,\nu}\in\C.
$$
It follows from Abel's theorem on double series that 
$$
	\sum_{k,\nu=0}^\infty c_{k,\nu}(w-\overline a)^k(z-a)^\nu
$$
converges to a function $F_a(z,w)$ holomorphic  in the polydisc $U_a\times U_a^*,$ which coincides with $f$ on $G_{U_a}.$  That is, $F_a(z,\overline z)=f(z),$ for $z\in U_a.$ We wish to define a function $F$ on the union of the polydiscs $U_a\times U_a^*, a\in \Omega$ by setting $F=F_a$ on each such polydisc. We must show that the function $F$ is well-defined. The intersection of two such polydiscs is the polydomain 
$(U_a\cap U_b)\times(U_a\cap U_b)^*$ which we abbreviate as $V\times V^*.$ The function $F_a-F_b$ is holomorphic on this polydomain and vanishes on the graph $G_V$ of $w=\overline z$ over $V.$ 

{\em Claim.} For an arbitrary domain $V\subset \C,$ every function $\Phi,$ holomorphic on a domain $W\subset\C^2$ containing the graph $G_V$ of $w=\overline z$ over $V,$ vanishes identcially on $W.$  

\begin{proof}
The zero set $Z$ of $\Phi$ is a complex analytic set of dimension at least $1.$ Suppose, to obtain a contradiction that the dimension of $Z$ is $1.$ Then the set of irregular points of $Z$ is of dimension zero and since $G_V$ is of dimension $1$ there is a regular point $(z_0,\overline z_0)$ in $Z\cap G_V.$ 
Let $X$ be the component of the regular set of $Z$ containing $(z_0,\overline z_0).$ For simplicity, we may assume that $z_0=0.$ Thus, $X$ is a complex submanifold (not necessarily closed) of $\C^2$ containing the origin and the graph $w=\overline z$ near the origin. Since $w=\overline z$ is a real linear subspace of $\C^2,$ the real tangent space of $X$ contains $w=\overline z$ and in particular the vector $(i,-i).$ But since $X$ is a complex submanifold its complex tangent space coincides with its real tangent space, so  $T_{\R}X_0=T_{\C}X_0$ contains the vector $(i,-i).$ It also contains $i$ times  this vector, which is $(-1,1).$ This is not in the linear subspace $w=\overline z$ of  $T_{\R}X_0$ so the latter is of real dimension at least $3.$ Again, since $T_{\R}X_0=T_{\C}X_0$ it folows that $T_{\C}X_0$ is of dimension $2$ so $X$ is of dimension $2$ at $0,$ which contradicts our assumption that $Z$ is of dimension $1.$ Thus, the zero set $Z$ of $\Phi$ is of dimension $2$ and so $\Phi$ vanishes identically.  
\end{proof} 

It follows from the claim that $F_a-F_b$ is identically zero  on the domain $V\times V^*.$
That is, $F_a-F_b$ is zero on  the intersection of the  two polydiscs. Thus, $F$ is well-defined on the union of all the polydiscs $U_a\times U_a^*, a\in\Omega.$  The function $F$ is holomorphic in a neighbourhood of the graph $G_\Omega$ and is equal to $f$ on $G_\Omega.$ That is, $F(z,\overline z)=f(z)$ for $z\in\Omega.$ The same uniqueness argument shows that if $F_1$ and $F_2$ are two holomorphic functions on neighbourhoods of $G_\Omega$ which coincide with $f$ on $G_\Omega,$ then $F_1=F_2$ on some neighbourhood of $G_\Omega.$ In this sense we may identify the space $CA(\Omega)$ with (germs of) functions holomorphic on the graph $G_\Omega$ over $\Omega$ of the function $w=\overline z$ in $\C^2.$ 

For the particular case that $f\in\mbox{PA}_q(\Omega)$ we can be more precise.  Since $f$ has the representation (\ref{eq1}) the function 
$$
	F(z,w)=\sum_{j=0}^{q-1}a_j(z)w^j, \quad z\in \Omega\times\C,
$$
if $X$ is an open connected neighbourhood (in $\Omega\times\C$) of the  graph $G_\Omega$ of $w=\overline z$ over $\Omega$  then $F$ is 
the unique holomorphic function on $X$ which coincides with $f$ on  $G_\Omega.$


With pointwise multiplication, the spaces $\mbox{PA}(\Omega)$ and $CA(\Omega)$ are algebras over the complex numbers $\C.$  The spaces $\mbox{PA}_q(\Omega)$ are not closed under usual pointwise multiplication. 

\end{definition}

\begin{definition} 
Let $X\subset\C$ be a compact subset.  We denote by $C^0(X)$ the Banach space of continuous functions on $X$ equipped with the sup-norm. For $q\in \Z_+,$ we denote by $P_q(X)$ the set of uniform limits on $X$, of continuous functions on $X$ that are $q$-analytic functions  on the interior $\mbox{int}(X)$ with polynomial analytic coefficients.
We denote by $A_q(X)$ the set of continuous functions on $X$ that are $q$-analytic functions on $\mbox{int}(X)$ (so if $X=\Omega$ is an open subset of $\C$ we have $\mbox{PA}_q(\Omega)=A_q(\Omega)$). 
In particular, for a subset $X\subset\Cn$ we denote by $A_1(X)$ the set of continuous functions on $X$ that are analytic on  $\mbox{int}(X).$ 

From Proposition \ref{balk206} we see that $A_q(X)$ is closed in  $C^0(X).$ We have the following Banach space inclusions: 
$$
		P_q(X) \subset A_q(X) \subset C^0(X).
$$

\end{definition}

\begin{definition}
	Given a $q$-analytic function $f=\sum_{j=0}^{q-1}a_j(z)\bar{z}^j,$
	we define the {\em exact order of $f$}
	to be the nonnegative integer $\max \{ j\colon a_j\not\equiv 0\}.$
\end{definition}

If $q\in \Z_+,$ and $f,g$ are two $q$-analytic functions on a domain $\Omega,$ such that the order $q$ is exact for both $f$ and $g$,
then $\partial_{\bar{z}}^q (fg)\equiv 0$ if and only if $q=1.$ Hence, for the case of $q>1$ we do not
have algebras of $q$-analytic functions. 

\begin{definition}
A {\em Banach algebra} is a commutative $\C$-algebra, $A$, with identity (a so called unital algebra)
that at the same time is a complex Banach space with respect to a norm $\norm{\cdot}$ that satisfies $\norm{xy}\leq \norm{x}\norm{y}$
for all $x,y\in A,$ and $\norm{1}=1$ (where $1$ in the left hand side denotes the identity in $A$ and $1$ in the right hand side denotes the 
identity in $\C$).
\end{definition}

\begin{definition}[See e.g.\ Stout \cite{stout}, Def. 1.2.6]
A {\em uniform algebra} on a compact Hausdorrf space $X$ is a uniformly closed, point-separating subalgebra with identity (so-called unital) of the algebra $C^0(X).$
\end{definition}
Recall that if $X$ is a compact Hausdorff space we say that
a subalgebra $A$ of $C^0(X)$
separates the points of $X$ if for any $p_1,p_2\in X$, with $p_1\neq p_2$, there exists $f\in A$ such
that $f(p_1)\neq f(p_2).$ 
Recall also that the complex version of the Stone-Weierstrass theorem (for denseness
in $C^0(X)$, in the norm topology)
includes the condition that the algebra $A$ satisfy 
$f\in A\Rightarrow \bar{f}\in A$). 
The standard example given by the uniform limits on $\{\abs{z}\leq 1\}$ of complex polynomials 
(which contains the constant $1$ and separates the points of $\{\abs{z}\leq 1\}$) 
does not satisfy $f\in A\Rightarrow \bar{f}\in A$ (indeed the uniform limit must be analytic on the interior of the disc). 
Similarly, the complex version of the Stone-Weierstrass theorem
is inapplicable to any kind of algebra that would contain a $q$-analytic function with a nonpolynomial analytic component but not its conjugate (which would of course not be a $k$-analytic function for any $k\in \Z_+$). 
We shall introduce a multiplication that allows for the construction of an algebra of $q$-analytic functions, for any fixed $q\in \Z_+,$ closed with respect to the sup-norm.
To do this we need some preparatory material from Balk \cite{balk}.

   \begin{proposition}[See Balk \cite{balk}, Cor. 4, p.206]\label{balk206}
     Let $\Omega\subset\C$ be a domain and let $q\in \Z_+.$
     If a sequence $\{f_j\}_{j\in \Z_+}$ of $q$-analytic functions converges locally uniformly inside $\Omega$
     then the limit function is also $q$-analytic in $\Omega.$
     \end{proposition}

       \begin{proposition}[See Balk \cite{balk}, Cor. 5, p.206]\label{balk206b}
          Let $\Omega\subset\C$ be a domain and let $q\in \Z_+.$
          If a sequence $\{f_j\}_{j\in \Z_+}$ of $q$-analytic functions converges locally uniformly to $f$ inside $\Omega$
          then the any pair of nonnegative integers $k,l$
           $\partial_{\bar{z}}^k\partial_z^l f_j(z)$ converges locally uniformly to 
           $\partial_{\bar{z}}^k\partial_z^l f(z)$ on $\Omega.$
          \end{proposition}

          A consequence is the following.

          \begin{corollary}\label{ahernbrunakonsekvens}
          Let $K\subset \C$ be a regular compact set  and let $q\in \Z_+.$
          Let $\{f_j\}_{j\in \Z_+}$ be a sequence of continuous functions on the closure $K$ such 	that 
          each $f_j$ is of the form $f_j(z)=\sum_{k=0}^{q-1}\bar{z}^k a_{j,k}(z),$
          $a_{j,k}\in A_1(K)$ (here $\bar{z}^j$ denotes the restriction of $\bar{z}^j$ to $K$) and
          satisfying $f_j\to f$ uniformly on $K$. Then $f$ is a
          continuous function on $K$, $q$-analytic on $\Omega,$ and has a representation of the form  $f(z)=\sum_{k=0}^{q-1}\bar{z}^k g_{k}(z).$ This representation is unique and 
          $g_{k}\in A_1(K)$, $k=0,\ldots,q-1.$
          \end{corollary}

          \begin{proof}
          That $f$ is continuous is immediate from the fact that $C^0(K)$ is closed under the sup-norm. 
          That $f$ is $q$-analytic on $\mbox{int}(K),$ follows from Proposition \ref{balk206}.

 By Proposition \ref{balk206b}
          we have for each $m=0,\ldots,q-1,$
          \begin{equation}\label{crocat}
          \partial_{\bar{z}}^m f_j\to \partial_{\bar{z}}^m f  \quad \mbox{locally uniformly on } \quad \mbox{int}(K).
          \end{equation}
          Also note that the fact that each of the analytic components $a_{j,k}(z)$ of $f_j$ has (since it belongs to $A_1(K)$ by definition) a continuous extension to $K,$ implies that the same holds true
          for $\partial_{\bar{z}}^m f_j=\sum_{k=0} a_{j,k}(z)(\partial_{\bar{z}}^m \bar{z}^k).$
In particular the normal limit
          $\frac{1}{(q-1)!}\partial_{\bar{z}}^{q-1} f_j\to g_{q-1}$, $j\to\infty,$ is an analytic function
          on $\mbox{int}(K)$ coinciding with $\lim_j a_{j,q-1}$ (see e.g.\ H\"ormander \cite{hormander}, Cor. 1.2.5, p.4), and has continuous extension to $K,$ thus we can identify $g_{q-1}$ as a member of $A_1(K)$.
          This in turn can be iterated for $f_j$ replaced by $(f_j-\bar{z}^{q-1}g_{q-1})$
          in order to obtain that $g_{q-2}$ is an analytic function
          on $\mbox{int}(K)$ coinciding with $\lim_j a_{j,q-2}$ with continuous extension to $K$, and so on.
          Hence $f$ can, on $\mbox{int}(K)$ be identified as
          the function $f(z)=\sum_{k=0}^{q-1} g_k(z)\bar{z}^k$ for analytic
          $g_k(z)$ such that each $g_k$ has continuous extension to $K.$ Such a representation is unique on $\mbox{int}(K)$ and since $\mbox{int}(K)$ is dense in $K,$ the coefficients extend uniquely to $K.$ 
          This completes the proof.
          \end{proof}


\section{The algebra $\mathbf{PA}_{q}(K)$}
\begin{definition}
For $K\subset\C$  compact,
let $A_1(K)$ be the uniform algebra of continuous  functions on $K,$ which are analytic on the interior of
$K.$  Let $\mbox{PA}_q(K)$ denote the additive group 
of functions $f$ on $K$ of the form $f(z)=\sum_{j=0}^{q-1}\bar{z}^j a_j(z),$
$a_j\in A_1(K)$ (here $\bar{z}^j$ denotes the restriction of $\bar{z}^j$ to $K$). In other words $\mbox{PA}_q(K)$ is the module over $A_1(K)$
generated by $\{1,\bar{z},\ldots,\bar{z}^{q-1}\}.$ In particular, $\mbox{PA}_1(K)=A_1(K).$ We define $\mbox{PA}(K)$ as the union of the $\mbox{PA}_q(K).$
\end{definition}

\begin{lemma}\label{C(K)}
Let $K\subset \C$ be a regular compact set. Then the natural mappings
$$
\begin{array}{llll}
(i) 	&	\mbox{PA}_q(K)|_{\mbox{int}(K)} & \longrightarrow &  A_q(\mbox{int}(K))\\
(ii)	&	\mbox{PA}(K)|_{\mbox{int}(K)} & \longrightarrow & PA(\mbox{int}(K))\\
(iii)	&	\mbox{PA}(K) & \longrightarrow & C(K)
\end{array}
$$
are injective 
\end{lemma}

\begin{proof}
The injectivity of (i)  follows immediately from Corollary \ref{ahernbrunakonsekvens}. 
The injectivity of the mapping $\mbox{PA}_q(K)\rightarrow C(K),$ follows since the analytic coefficients extend uniquely from $\mbox{int}(K).$ The injectivity of (ii) and (iii) follow since 
$\mbox{PA}(K)|_{\mbox{int}(K)}$ is the union of the $\mbox{PA}_q(K)|_{\mbox{int}(K)}$ and $\mbox{PA}(K)$ is the union of the $\mbox{PA}_q(K).$
\end{proof}

If $K\not=\overline{\mbox{int}(K)},$ this restriction mapping need not be injective. Even for $q=1,$ the restriction $A_1(K)\rightarrow A_1(\overline{\mbox{int}(K)})$ need not be injective. 

The space 
$$
	\mbox{PA}(K) = \bigcup_{q=1}^\infty\mbox{PA}_q(K)
$$
is  an algebra with respect to the  multiplication
$$
	\left( \sum_j\bar{z}^j a_j(z)\right)\left(\sum_k\bar{z}^k b_k(z)\right) =
		\sum_\ell\bar{z}^\ell\left(\sum_{j+k=\ell}a_j(z)b_k(z)\right)
$$
and we may form the quotient algebra $\mbox{PA}(K)/\langle\overline z^q\rangle,$ where $\langle \overline z^q\rangle$ is the principal ideal of $\mbox{PA}(K)$  generated by $\overline z^q.$  Then $\mbox{PA}(K)/\langle\overline z^q\rangle$ is an associative, commutative, and unital algebra over $\C$. The quotient algebra  $\mbox{PA}(K)/\langle\overline z^q\rangle$ consists of equivalence classes, each of which has a unique representative of the form $f(z)=\sum_{j=0}^{q-1}a_j(z)\overline{z}^j.$ We may thus make the cannonical identification 
$$
	\mbox{PA}_q(K) = \mbox{PA}(K)/\langle\overline z^q\rangle
$$
and endow $\mbox{PA}_q(K)$ with the multiplicative structure induced by this identification. In particular, we again find that  $\mbox{PA}_1(K)=A_1(K).$

We wish to introduce a norm on the algebra  $\mbox{PA}(K).$ To assure that the norm is well-defined, we shall restrict the class of compacta $K$ under consideration to those which are regular compacta.

As a consequence of Lemma \ref{C(K)}, we may consider the algebra  $\mbox{PA}_q(K)$  as a subalgebra of the Banach algebra $C(K)$ endowed with the norm $\|f\|_K=\max\{|f(z)|:z\in K\}$ and so as functions defined on $K.$ Thus, the algebra  $\mbox{PA}(K)$ is a normed subalgebra of  $C(K).$

\begin{definition}  
Let $K\subset\C$ be a regular compact set and let $q\in \Z_+.$
Denote by $\mathbf{PA}_{q}(K)=\mbox{PA}_{q,\diamond_q}(K)$ the space $\mbox{PA}_q(K)$  equipped with the usual addition but with multiplication replaced by  the operation, $\diamond_q ,$ of multiplication in  $\mbox{PA}(K)/\langle\overline z^q\rangle.$
That is, if $f,g\in \mathbf{PA}_q(K),$ we define $f\diamond_q g =f(z)\cdot g(z)\mbox{ mod }\langle\overline z^q\rangle.$ 
\end{definition}

 We may write this multiplication
\begin{equation}
	\diamond_q\colon  \mbox{PA}_q(K)\times \mbox{PA}_q(K) \to 				\mbox{PA}_q(K)
\end{equation}
more explicitely as follows: Let $f,g\in \mbox{PA}_q(K).$ By definition there exists  functions $a_j,b_j$ in $A_1(K)$ such that
$f(z)=\sum_{j=0}^{q-1}a_j(z)\bar{z}^j,$ $g(z)=\sum_{j=0}^{q-1}b_j(z)\bar{z}^j.$ Define 
\begin{equation}
	(f\diamond g)(z):= \sum_{j=0}^{q-1} \left(\sum_{k+l=j} a_k(z) b_l(z)\right)\bar{z}^j.
\end{equation}

Since $\mathbf{PA}_{q}(K)$ with the operation $\diamond$ can be identified with the quotient algebra  $\mbox{PA}(K)/\langle\overline z^q\rangle,$  we have an associative, commutative, algebra over $\C$.
Obviously, for $q=1$, $\diamond_1$ is the usual multiplication and $\mathbf{PA}_{q}(K)$ reduces to
the uniform algebra $A_1(K).$

The quotient algebra  $\mathbf{PA}_q(K)$ considered as a ring is not an integral domain. That is, there are zero divisors. For example, $\overline z^{q-1}\diamond\overline z=0.$ 
The quotient algebra $\mbox{PA}(K)/\langle\overline z^q\rangle$ has a natural quotient semi-norm: For an equivalence class $[f]\in \mathbf{PA}_q(K)$
$$
	\|[f]\|_q = \inf\{\|g\|_K:f\sim g\}.
$$
We may thus consider $(\mathbf{PA}_q(K),\|\cdot\|_q)$ as a semi-normed algebra.  

\begin{lemma}
	For $f\in \mbox{PA}(K),$ $\|[f]\|_q \le  \|f\|_K.$
\end{lemma}

\begin{proof}
$f\sim f.$ 
\end{proof}

Our assumption that $K$ is regular is important here. Without this assumption, the semi-norm  $\|\cdot\|_q$ may not be a norm. 
The semi-norm $\|\cdot\|_q$ on the quotient vector algebra $\mathbf{PA}_q(K)$ is a proper norm if and only if the ideal $\langle\overline z^q\rangle$ is topologically closed in $\mbox{PA}(K).$ 

Ideals of this kind are naturally topologically closed. Indeed, 
let $f_n(z)=\sum_{j=1}^{q-1}a_{n,j}(z)\overline{z}^j, n=1,2,\ldots$ be a sequence in $\langle\overline z^q\rangle$ converging to $f(z)=\sum_{j=1}^ma_j(z)\overline{z}^j$ in $\mbox{PA}(K),$ where $m\ge q$ and $a_m\not=0.$  We may write $f_n(z)=\sum_{j=1}^ma_{n,j}(z)\overline{z}^j,$ by setting $a_{n,j}=0,$ for $j\ge q.$ By our familiar argument of taking derivatives, we have that $a_{n,m}(z)\rightarrow a_m,$ as $n\rightarrow\infty.$ Therefore $a_m=0.$ This contradiction shows $m< q$ which proves that  $\langle\overline z^q\rangle$ is topologically closed.
A similar argument will verify that $\langle\overline z^q\rangle$ is topologically closed and therefore $(\mathbf{PA}_q(K),\|\cdot\|_q)$) is a normed algebra.

Since $\mathbf{PA}_{q}(K)\subseteq C^0(K),$ we can define on $\mathbf{PA}_{q}(K)$ another norm, namely the sup-norm
\begin{equation}
	\norm{f}_{\mathbf{PA}_{q}(K)}=\norm{f}_\infty=\sup_{z\in K} \abs{f(z)}=\|f\|_K.
\end{equation}
For simplicity, we shall denote this norm by  $\norm{\cdot }.$
By Corollary \ref{ahernbrunakonsekvens}
$\mathbf{PA}_{q}(K)$ is closed with respect to $\norm{\cdot }.$

However, we do not have a Banach algebra (note that we include automatically commutativity and the condition of the existence of a unit in the algebra, which is not always the case
in various literature, in such cases it is important to write unital and commutative Banach algebra)
because the condition $\norm{fg}\leq \norm{f}\norm{g}$ will not continue to hold true when we replace the multiplication by $\diamond_q.$
\begin{example}\label{banachcounter}
Let $K:=\{z\colon \abs{z-3/4}\leq 1/4\}.$
 There exists $f,g\in \mbox{PA}_2(K)$ such that $\norm{ f}>\norm{f}\cdot \norm{g}.$ Indeed, let $f(z):=z\bar{z},$ $g(z):=1-z\bar{z}.$
 Then $f\diamond_2 g=z\bar{z},$ so that $\norm{f}=1,$ $\norm{g}=1-1/4=3/4,$ $\norm{f\diamond_2 g}=1>1\cdot 3/4=\norm{f}\cdot \norm{g}$.
\end{example}

\begin{remark}\label{innanvarje}
Due to (the counter-) Example \ref{banachcounter} it makes sense to consider replacing, in the definition of uniform algebra, ``Banach algebra" by ``algebra with the standard addition, that is a Banach space under the sup-norm". 
\end{remark}


\section{A generalized notion of uniform algebras}
We begin with a remark.
\begin{remark}
Note that the definition of a uniform algebra, $A$, requires that the  algebra be a subalgebra of $C(X)$, so formally this would require
that the multiplication in $A$ coincide with that in $C(X).$ We do however have the following.
Let $X\subset \C$ be a regular compact set with $\mbox{int}(K)$ connected. A function $P\in \mbox{PA}_q(\mbox{int}(X))$ is a polyanalytic polynomial
(in the sense that its analytic components are complex polynomials) if and only if
$P(z)$ is a ``polyanalytic polynomial" with respect to the multiplication in 
$\mathbf{PA}_{q}(\mbox{int}(X))$, in the sense that for some $N,M\in \Z_+,$ and complex constants $c_{jk}$
\begin{equation}\label{parepbf}
P(z)=\sum_{j=0}^N \left(\sum_{k=0}^{M} c_{jk} \underbrace{(z \diamond_q \cdots \diamond_q z)}_{k-\mbox{times}}  \right) \underbrace{(\bar{z} \diamond_q \cdots \diamond_q \bar{z})}_{j-\mbox{times}}.
\end{equation}
This follows immediately from the fact that 
\begin{equation}
\underbrace{(z \diamond_q \cdots \diamond_q z)}_{k-\mbox{times}} =z^k,\quad \underbrace{(\bar{z} \diamond_q \cdots \diamond_q \bar{z})}_{j-\mbox{times}}=
	\bar{z}^j \mbox{ mod}\,\bar{z}^q
\end{equation}

Denote by  $C^{q-1}(K)$ the set of functions 
that are $q$-analytic on $\mbox{int}(K)$ such that their analytic components have
continuous extensions to $K$, in the sense that they belong to $A_1(K)$ and where
we can extend $\diamond$ using existence and uniqueness of boundary values. 
Denote by $\mathbf{P}_q(X)$ the set of uniform limits on $X$, of $C^{q-1}(X)$-functions having the representation in Eqn.(\ref{parepbf}). 
Then, if $K$ is regular and $\mbox{int}(X)$ is connected, 
\begin{equation}
P_q(X)\cap C^{q-1}(X)=\mathbf{P}_q(X).
\end{equation}
Similarly, it is easily realized that if $K$ is regular, then
\begin{equation}
A_q(X)\cap C^{q-1}(X)=\mathbf{PA}_q(X).
\end{equation}
\end{remark}

\begin{definition}[Uniform algebra in the generalized sense]
A {\em uniform algebra in the generalized sense}, $A$, on a compact Hausdorrf space $X$ is a uniformly closed, point-separating (commutative and unital) algebra  over $\C$, which is a subset of $C(X)$ (but not necessarily a subalgebra of  $C(X)$)  with stadard identity and with the standard addition, that is a Banach space under the sup-norm
such that $A\subset C(X)$ (i.e.\ we do not require that the multiplication in $A$ coincide with that in $C(X)$).
\end{definition}

\begin{observation}
For a regular compact set $K\subset \C,$ we have  that
$\mathbf{PA}_{q}(K)$ is a uniform algebra in the generalized sense.
\end{observation}
\begin{proof}
We have already proved that 
$\mathbf{PA}_{q}(K)$ is a unital commutative 
algebra contained in $C(X)$, which is closed with respect to the sup-norm and contains the constants. We
need to verify that it separates points of $K$, but this follows from the fact
that it contains the set of functions continuous on $K$ and analytic on $\mbox{int}(K)$ and and that is sufficient to separate points. This proves the observation.
\end{proof}


\section{Some observations regarding the spectrum}

In this section we consider spectral analysis  in the theory of polyanalytic functions.
The spectrum as we shall see can be identified as a subset of the complex dual space equipped with the so-called weak-star topology (see below).
First we recall some preliminaries.
\begin{definition}
Let $A$ be a  (commutative unital) Banach algebra. Denote by $A^{-1}$ the multiplicative group of invertible elements
in $A$ (i.e.\ $f\in A^{-1}$ if there exists $g\in A$ such that with the multiplication in $A$ we have $f\cdot g=1$).
We denote by $f^{-1}$ the inverse element of $f$.
A complex number is said to belong to the {\em resolvent} of $f\in A$ if $\lambda\cdot 1 -f$ is invertible. 
The set of complex $\lambda$ for which $\lambda\cdot 1 -f$ is not invertible, is called the {\em spectrum} of $f$ (in $A$)
and is denoted $\sigma_A(f).$ 
\end{definition}

For our purposes (see Remark \ref{innanvarje}), it makes sense to consider the following.
\begin{definition}
Let $A$ be a uniform algebra in the generalized sense. Denote by $A^{-1}$ the multiplicative group of invertible elements
in $A$ (i.e.\ $f\in A^{-1}$ if there exists $g\in A$ such that with the multiplication in $A$ we have $f\cdot g=1$).
We denote by $f^{-1}$ the inverse element of $f$.
A complex number is said to belong to the {\em resolvent}\index{Resolvent} of $f\in A$ if $\lambda\cdot 1 -f$ is invertible. 
The set of complex $\lambda$ for which $\lambda\cdot 1 -f$ is not invertible, is called the {\em spectrum} of $f$ (in $A$)
and is denoted $\sigma_A(f).$ 
\end{definition}

\begin{lemma}\label{XYZ}
Let $X, Y$ and $Z$ be Banach spaces with $Y$ a closed subspace of $Z$ and let $F:X\rightarrow Z$ be analytic. If $F(X)\subset Y,$ then $F:X\rightarrow Y$ is also analytic. 
\end{lemma}

\begin{proof}
It is sufficient to show that, for each $\ell\in Y^*,$ the composition $\ell\circ F$ is analytic. Let $L\in Z^*$ be a continuous extension of $\ell.$ Since $F:X\rightarrow Z$ is analytic,  $L\circ F$ is analytic. But $\ell\circ F= L\circ F.$
\end{proof}

\begin{lemma}\label{Fg}
Let $G$ be an open subset of $\C,$ let $F:G\longrightarrow C(K)$ be analytic and let $g\in C(K).$ Then the mapping $G\rightarrow C(K)$ given by $\lambda\mapsto F(\lambda)\cdot g$ is also analytic. 
\end{lemma}

\begin{proof}
Since $F$ is analytic it has a local representation as a power series 
$$
	F(\lambda) = \sum_{k=0}^\infty (\lambda-\lambda_0)^k\alpha_k, \quad \alpha_k\in C(K). 
$$
Thus, 
$$
	F(\lambda)\cdot g = \sum_{k=0}^\infty (\lambda-\lambda_0)^k\alpha_k\cdot g, \quad \alpha_k\cdot g\in C(K). 
$$
Thus, $\lambda\mapsto F(\lambda)\cdot g$ is also analytic. 
\end{proof}

\begin{proposition}
	Consider the uniform algebra in the generalized sense, $\mathbf{PA}_q(K)$, for a fixed $q\in \Z_+,$ and a regular compact set $K\subset\C.$ 
Let $f(z)$ have on $U$ the representation $\sum_{j=0}^{q-1}a_0(z)\bar{z}^j$, where each $a_j(z)$ is analytic and has continuous extension to $K.$
	Then $\sigma_{\mathbf{PA}_q(K)}(f)\subseteq a_0(K)$.
	In particular the spectrum is closed and bounded, i.e.\ compact.
	Furthermore, 
	$(\lambda -f)^{-1}$ is an analytic function of $\lambda$
	on $\{\C\setminus a_0(K)\}.$
\end{proposition}

\begin{proof}
Let $f\in \mathbf{PA}_q(K)$ with representation  given by $f(z)=\sum_{j=0}^{q-1}a_j(z)\bar{z}^j$,  on $U.$ 
 
From the general theory of uniform algebras, the set of $\lambda\in\C$ where $(\lambda-a_0)$ is invertible in $A_1(K)$ is open and contains $\C\setminus a_0(K).$ That is, the resolvant $R_0$ of $a_0$ is open and the spectrum $\sigma_0$ of $a_0$ is contained in $a_0(K).$ Moreover, the mapping $\C\setminus a_0(K)\rightarrow A_1(K),$ given by $\lambda\mapsto (\lambda-a_0)^{-1}$ is analytic. 

Supoose $\lambda\in\C$ is such that $(\lambda-f)$ is invertible in  $\mathbf{PA}_q(K).$ 
If  $h(\lambda)=(\lambda-f)^{-1}\in \mathbf{PA}_q(K)$, with representation on $U$ given by 
$$h(\lambda)(z)=\sum_{j=0}^{q-1}c_{\lambda,j}(z)\bar{z}^j,$$ 
then, since  $(\lambda-f)h(\lambda)=1,$ from the uniqueness of the analytic coefficients:
$$
	(\lambda-a_0)c_{\lambda,0}=1, 
$$
$$
	(\lambda-a_0(z))c_{\lambda,j}-\sum_{i=1}^ja_ic_{\lambda,j-i}=0, \quad j=1,\ldots,q-1.
$$
Notice that we have shown that if $(\lambda-f)$ is invertible in $\mathbf{PA}_q(K),$ then  $\lambda$ is in the resolvant $R_0$ of $a_0\in A_1(K).$ Conversely, if $\lambda\in R_0,$ then the function $c_{\lambda,0}$ and, 
for $j=1,\ldots,q-1,$ the functions $c_{\lambda,j}:\C\setminus a_0(K)\rightarrow A_1(K),$ recursively defined by
$$
	c_{\lambda,j} = \sum_{i=1}^j(\lambda-a_0)^{-1}a_ic_{\lambda,j-i},
$$
are well-defined. We have shown that the spectrum of $f$ in $\mathbf{PA}_q(K)$ is the same as the spectrum of $a_0$ in $A_1(K).$ Thus, the spectrum of $f$ is compact and contained in $a_0(K).$  

It also follows that the functions $\lambda\mapsto c_{\lambda,j}$ are  analytic  on $\C\setminus a_0(K).$  
We have found the analytic coefficients of $h(\lambda).$ It follows from Lemma \ref{Fg} that for $j=0,\ldots,q-1,$ the mappings $\C\setminus a_0(K)\rightarrow C(K),$ given by $\lambda\mapsto c_{\lambda,j}\overline z^j$ are analytic and by Lemma \ref{XYZ} that they are also analytic to  $\mathbf{PA}_q(K).$ Hence, the mapping $\lambda\mapsto (\lambda-f)^{-1},$ as a finite sum of the analytic mappings $\lambda\mapsto  c_{\lambda,j}\overline z^j,$ is also analytic from $\C\setminus a_0(K)$ to  $\mathbf{PA}_q(K).$

This concludes the proof.

\end{proof}

Note that 
$\lambda\in \sigma_{\mathbf{PA}_q(K)}(f)$ implies that
$\abs{\lambda}\leq \max_{z\in K}\abs{a_0(z)}=\norm{a_0}.$ 
For a $q$-analytic function $f(z)=\sum_{j=0}^{q-1}a_j(z)\bar{z}^j$ on a domain $\Omega,$
we have an associated function of two complex variables
\begin{equation}
F(z,w):=\sum_{j=0}^{q-1}a_j(z)w^j, \quad (z,w)\in \Omega\times\C. 
\end{equation}
This implies
$a_0(z)=F(z,0).$
We immediately have the following.

\begin{corollary}
	Let $q\in \Z_+$ and for a regular
compact set $K\subset\C,$ 
let $f\in \mathbf{PA}_q(K)$. Then
	$\sigma_{\mathbf{PA}_q(K)}(f)\subseteq \{\abs{z}\leq \norm{F(z,0)}\}.$
\end{corollary}

In the case $q=1$ we have $f(z)=a_0(z)$ so clearly
$\lambda\in \sigma_{\mathbf{PA}_1(K)}(f)$ implies that
$\abs{\lambda}\leq \max_{z\in K}\abs{f(z)}=\norm{f}.$


\section{Characters}

\begin{definition}
A character $\varphi$ of a Banach algebra $F$
is an homomorphism $\varphi:F\rightarrow\C$ which is non-trivial (meaning that it is not identically $0$). 
We denote the set of characters of $F$ by Char$(F).$ We note that characters are continuous.
\end{definition}

We are interested in the analogous case when $\varphi$ is a multiplicative linear functional on $\mathbf{PA}_{q}(K)$ where 
the multiplication 
is given by $\diamond_q.$
In light of Remark \ref{innanvarje} our analogue of the previous definition will be the following.

\begin{definition} Let  $F(+,\bullet)$ be  a uniform algebra in the generalized sense, (where $\bullet$ denotes the multiplication in $F$). 
A character $\varphi$ of $F$,
is a non-trivial continuous linear functional on $F$ which is multiplicative. 
Thus, $\varphi\in F^\prime$ (where $F^\prime$ denotes the topological dual of the Banach space $F$) such that for all $f,g\in F,$
$\varphi(f \bullet g)=\varphi(f)\varphi(g).$ We denote the set of characters of $F$ by Char$(F).$ Note that $\varphi(1)=1.$
\end{definition}

Denote by $(z,w)$ the Euclidean coordinate in $\C^2$ and for a compact $S\subset\C^2,$ denote 
by $A_1(S)$ the set of functions, 
holomorphic for $(z,w)\in \mbox{int}S$ and continuous on $S.$
For a subset $K\subset \C$ denote by $K^*:=\{\bar{z} \colon z\in K\}\subset \C.$

Denote by $(z,w)$ the Euclidean coordinate in $\C^2$ and for a compact $S\subset\C^2,$ denote 
by $A_1(S)$ the set of functions,  continuous on $S$ and
analytic for $(z,w)\in \mbox{int}(S).$  Characters of spaces of type $A_1(S)$ are discussed for example in \cite{stout}. 

In particular 
for a regular compact subset $K\subset \C$ denote by $K^*:=\{\bar{z} \colon z\in K\}\subset \C$    and set $S=K\times K^*.$ Suppose $f\in PA_q(K).$ Then, 
\begin{equation}\label{f}
	f(z) = \sum_{j=0}^{q-1}a_j(z)\overline z^j, \quad a_j\in A_1(K). 
\end{equation}
The  function $F\in A_1(K\times \C)$ defined as 
\begin{equation}\label{F}
	F(z,w) =  \sum_{j=0}^{q-1}a_j(z)w^j, \quad (z,w)\in K\times \C \subset \C^2,
\end{equation}
may be considered as an extension of $f$ in the sense that it coincides with $f$ when restricted to the graph of $w=\overline z$ over $K.$ 
This yields an injection of the space  $PA_q(K)$ to the space  $A_1(K\times\C).$ Moreover, we have seen in the introductory paragraphs that $F$ is the only function analytic on $ \mbox{int}(K)\times\C$ whose restriction to the graph of $w=\overline z$ over $ \mbox{int}(K)$ coincides with $f.$ Since $ \mbox{int}(S)= \mbox{int}(K)\times \mbox{int}(K^*),$ and $K$ is regular, it follows that  $\mbox{int}(S)$ is dense in $S$ and so
the function $F$ is the only function in $A_1(S),$ whose restriction to the graph of $w=\overline z$ over $K$ coincides with $f.$ We have shown that we can consider $PA_q(K)$ as a subspace of $A_1(S).$ 

Let $(w^q)$ be the ideal in $A_1(K\times\C)$ generated by the function $w^q$ considered as an element of $A_1(K\times\C).$ Although the natural projection $A_1(K\times\C)\rightarrow A_1(K\times\C)/(w^q)$ is not an injection, the composition $PA_q(K)\rightarrow A_1(K\times\C)\rightarrow A_1(S)/(w^q)$ is an injection, so we may in this sense associate to a function $f\in PA_q(K)$ a unique member of  $A_1(K\times\C)/(w^q).$ To put it differently, the projection  $A_1(K\times\C)\rightarrow A_1(K\times\C)/(w^q)$ restricted to  $PA_q(K)$ as a subset of $A_1(K\times\C)$ is an injection. 

In the same manner, we may assiciate to each function $f\in PA_q(K)$ a unique member of  $A_1(S)/(w^q).$ That is,  the projection  $A_1(S)\rightarrow A_1(S)/(w^q)$ restricted to  $PA_q(K)$ as a subset of $A_1(S)$ is an injection.

If $A_1(S)$ is defined using the usual multiplication and the topology of uniform convergence, then it is a uniform algebra, but  the ideal $(w^q)$ need not be closed, even in the simpler situation where we consider $A_1(Q),$ with $Q$ a compact set in $\C= \C_w.$ The following example was given to us by Thomas Ransford. Consider the case where $Q$ is the disc with center $1$ and radius $1.$  We claim that

(i) the closed ideal generated by $w$ (i.e. the closure of $(w)$) consists of all functions in $A_1(Q)$ that vanish at $0.$

(ii) there exists a function in $A(Q)$ vanishing at $0$ that is not in $(w).$
Clearly (i) and (ii) together show that $(w)$ is not closed in $A(Q).$

Proof of (i). Let $f$ be in $A(Q)$ with $f(0)=0.$  `Dilating', we can approximate $f$ by a function $g$ holomorphic on an open neighbourhood of $Q.$ The new function $g$ need not vanish at $0,$ but it must be small there, so, adding a small constant if necessary, we can suppose that $g(0)=0.$ Since $0$ is now in the interior of the set where $g$ is holomorphic, we can write $g(w)=h(w)\cdot w,$ where $h$ is also holomorphic in a neighbourhood of $Q.$ In particular, $g$ lies in $(w).$

Proof of (ii). Consider $f(w)=\sqrt w.$ This belongs to $A(Q)$ and vanishes at $0.$ But $f$ is not in $(w)$ because $|f(w)|/|w|$ tends to infinity as $w$ tends to $0$ in 
$Q.$

We would like to describe the space Char$A_1(S)$ with the help of polynomial approximation. 
For a compact set $X\subset\C^n,$ denote by $P(X)$ the uniform limites on $X$ of polynomials. We always have $P(X)\subset A_1(X).$ 
In general, the problem of describing compact sets $X\subset\C^2,$ for which $P(X)=A_1(X)$ is extremely difficult.  Even for the case that $X$ is a product $K_1\times K_2$ and $P(K_i)=A(K_i), i=1,2,$ it is not always the case that $P(K_1\times K_2) = A_1(K_1\times K_2).$ For a simple counter example, we may take $K_1$ to be a singleton and $K_2$ to be a closed disc. 
The function $\overline w$ is in $A_1(K_1\times K_2)$ but not in $P(K_1\times K_2).$ However, for our set $S=K\times K^*$ we may apply the following result.

\begin{theorem}\label{product}
If $K_1$ and $K_2$ are regular compact subsets of $\C$ with connected complements and $K=K_1\times K_2,$ then 
$P(K) = A_1(K).$ 
\end{theorem}

\begin{proof}
Since $\C\setminus K_i$ are connected, it follows from Mergelyan's theorem that $P(K_i)=A_1(K_i)$ and consequently $R(K_i)=A_1(K_i),$ where $R(K_i)$ denotes uniform limits on $K_i$ of rational functions having no poles on $K_i.$ Taking into account that what we have denoted $A_1(K)$ is usually denoted by $A(K)$ in the litterature, it follows from Corollary 3.6 and Theorem 4.6 in \cite{FGMN} that $$\overline{\mathcal O}(K)=A_1(K),$$
where $\overline{\mathcal O}(K)$ denotes uniform limits of functions holomorphic on $K.$ Since $K_1$ and $K_2$ have connected complements they are polynomially convex and hence their product  $K$ is also polynomially convex. That is, $K=\widehat K,$ where $\widehat K$ denotes the polynomially convex hull of $K.$ 
By the Oka-Weil Theorem 
$$\overline{\mathcal O}(K)\subset P(K).$$ 
Since $P(K)\subset A(K),$ for every compact set, the theorem follows. 

\end{proof}

\begin{theorem}[See e.g.\ Stout \cite{stout}, Thm 1.2.9]\label{Stout}
If $K\subset\Cn$ is compact then every character of $P(K)$ is of the form $f\mapsto \hat{f}(z)$ for 
a unique $z\in \widehat{K}.$
\end{theorem}

\begin{corollary}
If $K\subset \C$ is a regular compact set with connected complement and $S=K\times K^*,$ then evey character of $A_1(S)$ is of the form $f\mapsto \hat f(z,w),$ for a unique $(z,w)\in S.$ 
\end{corollary}

\begin{proof}
This follows immediately from Theorems \ref{product} and \ref{Stout} since $S=\widehat S.$ 
\end{proof}

This completely characterizes the space Char$A_1(S)$ as the evaluations at points  of $S.$

In a more general setting, if $Y$ is a set and $T$ a topological space and $F$
a family of functions $Y\to T.$ The {\em weak topology}
on $Y$ induced by $F$ is the smallest (coarsest) topology, $\tau$, on $Y$
for which the functions of $F$ are continuous. Then $\tau$ is generated by (i.e.\ has neighbourhood basis
at each point that belongs to) the sets
\begin{equation}
\{f^{-1}(U)\colon f\in F,U\mbox{ open in }T\}, \quad f^{-1}(U):=\{y\in Y\colon f(y)\in U\}.
\end{equation}
It follows for   a net $\{x_\alpha\}_\alpha$ that $\lim_{\alpha} x_\alpha =x$ weakly if and only if $\lim_\alpha f(x_\alpha)=f(x)$ for each $f\in F.$ 
If $T$ is Hausdorff and $F$ separates the points of $Y$ then the weak topology is Hausdorff.

\begin{definition}
Let $Y$ be a complex Banach space. For each $f\in Y$ denote by $\hat{f}$ the function on the  dual $Y'$
defined by \begin{equation}\label{gelftrans}
\hat{f}(\phi)=\phi(f). 
\end{equation}
The weak-star topology on $Y'$ is the weak topology on $Y'$ induced by the family $\{ \hat{f}\colon f\in Y\}.$
\end{definition}

\begin{proposition}\label{characthausdorf}
The weak-star topology of the dual $A'$ of a Banach space $A$, is Hausdorff.
\end{proposition}
\begin{proof}
Let $\phi_1,\phi_2\in A'$.
If $\phi_1\neq \phi_2$ there exists $f\in A$ such that $\phi_1(f)\neq \phi_2(A).$
Hence $\hat{f}(\phi_1)\neq \hat{f}(\phi_2)$ which implies that the functions $\{\hat{f}\colon f\in A\}$ separate the points of $A'$.
Since $\C$ is Hausdorff the fact that the functions $\{\hat{f}\colon f\in A\}$ separate the points of $A'$
is sufficient for $A'$ (with the weak-star topology) to be Hausdorff.
\end{proof}

\begin{definition}
The spectrum $\Sigma(A)$ of a generalized uniform algebra $A$ on a compact Hausdorff
space $X$ is the set of all characters of $A$ equipped with the relative weak-star topology, That is, 
since $\Sigma(A)\subset A^\prime,$ we endow $\Sigma(A)$ with the weak-star topology inherited from $A^\prime.$  
Thus, a net $\{\phi_\iota\}_\iota$ in $\Sigma(A)$ converges to $\phi\in \Sigma(A)$ if for each $f\in A$
the net $\{\phi_\iota (f)\}_\iota$ converges to $\phi(f).$
\end{definition}

A neighbourhood basis of $\phi\in \Sigma(A)$ is given by sets of the form
\begin{equation}
\{ \psi\in \Sigma(A)\colon \abs{\phi(f_k)-\psi(f_k)}<\epsilon, k=1,\ldots,r\},
\end{equation}
where $\epsilon$ ranges over $(0,\infty),$ $r$ ranges over $\Z_+$ and $f_j$ ranges over $A$.

\begin{proposition}
For a uniform algebra $A$  in the generalized sense on a compact Hausdorff
space $X,$
the spectrum $\Sigma(A)$  is a compact Hausdorff space.
\end{proposition}

\begin{proof}
By Proposition \ref{characthausdorf} the space $A^\prime$ with the weak-star topology is Hausdorff, so the space of characters is also Hausdorff since a subspace of a Hausdorff space is always Hausdorff.  We must prove compactness.
Denote for each $f\in A$, by $\C^f$, a copy of $\C$ and denote by $\C^A$ the Cartesian product $\pi_{f\in A}\C^f.$ 
Set 
$$
	\Phi:\Sigma(A)\longrightarrow \C^A, \quad    \phi\longmapsto \Phi(\phi), 
	\quad \mbox{where} \quad  \Phi(\phi)_f = \phi(f).
$$ 
By the definition of the topologies
$\Phi$ is injective and continuous. Furthermore
$\Phi(\Sigma(A))$ is a subset of $\Pi_{f\in A} \{\zeta\colon \abs{\zeta}\leq\norm{\phi}\norm{f}_X\}$ which is compact since
by Tychonoff's theorem the product of compact sets is compact
in the product topology. 
Let $\{z_f:f\in A\}$ be a limit point of $\Phi(\Sigma(A))$ i.e.\ there exists a net $\{\phi_\alpha\}_\alpha$ in $\Sigma(A)$,
$\phi_\alpha(f)\to z_f$ for each $f\in A.$ Define $\varphi:A\to \C,$ $\varphi(f):=z_f.$ Then $\varphi(1)=1$ and $\varphi$ is $\C$-linear and multiplicative
so $\varphi\in \Sigma(f)$. Thus Ran$(\Phi)$ is closed in a compact space so it is itself a compact space.
If $\{\Phi(\phi_\alpha)\}_\alpha$ converges to $\Phi(\varphi),$ then $\phi_\alpha(f)\to \varphi(f)$ thus $\phi_\alpha\to \varphi$ in $\Sigma(A).$ 
Hence $\Phi^{-1}$ is continuous proving that $\Phi$ is a homeomorphism which in turn implies that
the inverse image $\Sigma(A)$ of $\Phi(\Sigma(A))$ is compact. This completes the proof.
\end{proof}

\bibliographystyle{amsplain}

\end{document}